\documentclass{amsart}
\usepackage{amssymb,latexsym,amsmath,amsthm,enumitem,mathrsfs,geometry}
\usepackage{fullpage}
\usepackage{fancyhdr}
\usepackage{color}
\usepackage{stmaryrd}
\usepackage{mathrsfs}
\usepackage{hyperref}
\usepackage{lineno}
\usepackage{diagbox}
\usepackage{array}
\usepackage{tikz}
\usetikzlibrary{arrows}
\usetikzlibrary{positioning}
\usepackage{float}
\usepackage{bbm}
\usepackage{cleveref}
\usepackage{dsfont}
\usepackage{graphicx}
\usepackage{soul}

\newtheorem*{thm*}{Theorem}

\newcommand{\beq}{\begin{equation}}
\newcommand{\eeq}{\end{equation}}

\newtheorem{theorem}{Theorem}
\newtheorem{lem}{Lemma}

\newtheorem*{theorem*}{Theorem}
\newtheorem*{conjecture*}{Conjecture}
\theoremstyle{remark}
\newtheorem*{remark}{Remark}
\definecolor{pink}{rgb}{1,.2,.6}
\definecolor{orange}{rgb}{0.7,0.3,0}
\definecolor{blue}{rgb}{.2,.6,.75}
\definecolor{green}{rgb}{.4,.7,.4}
\definecolor{purple}{RGB}{127,0,255}

\numberwithin{equation}{section}

\begin{document}
\title[Unconditional Montgomery Pair Correlation]{An unconditional Montgomery Theorem for Pair Correlation of Zeros of the Riemann Zeta Function}

\author[Baluyot]{Siegfred Alan C. Baluyot}
\address{American Institute of Mathematics}
\email{sbaluyot@aimath.org}

\author[Goldston]{Daniel Alan Goldston}
\address{Department of Mathematics and Statistics, San Jose State University}
\email{daniel.goldston@sjsu.edu}

\author[Suriajaya]{Ade Irma Suriajaya}
\address{Faculty of Mathematics, Kyushu University}
\email{adeirmasuriajaya@math.kyushu-u.ac.jp}

\author[Turnage-Butterbaugh]{Caroline L. Turnage-Butterbaugh}
\address{Carleton College}
\email{cturnageb@carleton.edu}


\dedicatory{Dedicated to Henryk Iwaniec on the occasion of his 75th birthday}

\begin{abstract}
Assuming the Riemann Hypothesis (RH), Montgomery proved a theorem concerning pair correlation of zeros of the Riemann zeta-function. One consequence of this theorem is that, assuming RH, at least $67.9\%$ of the nontrivial zeros are simple. Here we obtain an unconditional form of Montgomery's theorem and show how to apply it to prove the following result on simple zeros: Assuming all the zeros $\rho=\beta+i\gamma$ of the Riemann zeta-function such that $T^{3/8}<\gamma\le T$ satisfy $|\beta-1/2|<1/(2\log T)$,
then, as $T$ tends to infinity, at least $61.7\%$ of these zeros are simple. The method of proof neither requires nor provides any information on whether any of these zeros are on or not on the critical line where $\beta=1/2$. We also obtain the same result under the weaker assumption of a strong zero-density hypothesis. 
\end{abstract}
\date{\today}

\maketitle

\section{Introduction and Statement of Results}

Let $\rho=\beta +i\gamma$ denote a nontrivial zero of the Riemann zeta-function $\zeta(s)$ with $\beta,\gamma \in \mathbb{R}$, that is, a zero satisfying $\beta>0$. 
The Riemann Hypothesis (RH) states that $\beta=1/2$ for all $\rho$. To study the pair correlation of zeros of the zeta-function, Montgomery \cite{Montgomery73} assumed RH and considered, for $x>0$ and $T\ge3$, the sum
\begin{equation}\label{MonF} \sum_{ 0<\gamma,\gamma' \le T} x^{i(\gamma-\gamma')} \frac{4}{4+(\gamma-\gamma')^2}.
\end{equation}
The first goal of this paper is to generalize Montgomery's pair correlation method so that it is unconditional. To this end, define, for $x>0$ and $T\ge3$,
\begin{equation}\label{F(x,T)}
F(x,T) := \sum_{\substack{\rho, \rho' \\ 0<\gamma,\gamma' \le T}} x^{\rho-\rho'}w(\rho-\rho'), \qquad \text{where} \qquad w(u) := \frac{4}{4 - u^2},
\end{equation}
where here and throughout the paper, zeros are counted with multiplicity. Note that if RH holds then \eqref{F(x,T)} agrees with \eqref{MonF}.

Following Montgomery, we normalize $F(x,T)$ by defining, for real $\alpha$,
\begin{equation}\label{MonF(alpha)} F(\alpha) := \left(\frac{T}{2\pi}\log T\right)^{-1} F(T^\alpha,T) = \left(\frac{T}{2\pi}\log T\right)^{-1} \sum_{\substack{\rho, \rho' \\ 0<\gamma,\gamma' \le T}} T^{\alpha(\rho-\rho')}w(\rho-\rho').\end{equation}
The first result of this paper is the following unconditional theorem.
\begin{theorem}\label{thm1}
The function $F(\alpha)$ is real, even, and nonnegative. Moreover, as
$T \to \infty$, we have
\begin{equation}\label{Fasymp} F(\alpha) =
 T^{-2\alpha}(\log T + O(1))+ \alpha + O\left(\frac1{\sqrt{\log T}}\right) 
\end{equation}
uniformly for $0\le \alpha \le 1$.
\end{theorem}
\Cref{thm1} is nearly identical to Montgomery's theorem in \cite{Montgomery73} and \cite[Lemma 8]{GM87} except it does not assume RH, and it includes the improvements from \cite[Lemma 8]{GM87} where \eqref{Fasymp} holds up to $\alpha=1$ with explicit error terms. The proof is also nearly identical. See also \cite{IK04}. A simple application of Theorem \ref{thm1} concerns the multiplicities of the zeros of $\zeta(s)$. We use a slight modification of a kernel due to Tsang \cite{Tsang3} to prove the following result.

\begin{theorem}\label{thm2} Suppose that all the zeros $\rho = \beta +i\gamma$ of the Riemann zeta-function with $T^{3/8}<\gamma\le T$ lie within the thin box
\begin{equation}\label{strip} \frac1{2} -\frac1{2\log T}< \beta < \frac1{2} + \frac1{2\log T}. \end{equation}
Then for any sufficiently large $T>0$, at least $61.7 \%$ of the nontrivial zeros are simple. \end{theorem}

\begin{remark}
The pair correlation method developed in this paper neither requires nor provides any information as to whether or not the nontrivial zeros of $\zeta(s)$ satisfy $\beta=1/2$.
\end{remark}

There are many results concerning the proportion of nontrivial zeros of the Riemann zeta-function that are simple.
Pratt, Robles, Zaharescu, and Zeindler \cite{Pra20} have proved that more than $41.7\%$ of the zeros are on the critical line, and also that more than $40.7\%$ of the zeros are on the critical line and simple. Conrey, Iwaniec, and Soundararajan \cite{CIS13} have proved that more than $14/25=56\%$ of the nontrivial zeros of all Dirichlet $L$-functions are on the critical line and are simple. Going back to the case of the Riemann zeta-function, assuming RH, Montgomery \cite{Montgomery73} obtained from \Cref{thm1} that more than $2/3 = 66.\overline{6}\%$ of the zeros are simple, and soon after, Montgomery and Taylor \cite{Montgomery74} improved this to prove more than $ 67.2\%$ of the zeros are simple. Recently Chirre, Gonçalves, and de Laat \cite{CGL2020} obtained by this method $67.9\%$. By a mollifier method Conrey, Ghosh, and Gonek \cite{CGG98} showed on RH and an additional hypothesis that at least $19/27=70.3\overline{703}\%$ of the zeros are simple, and later Bui and Heath-Brown \cite{BuiHB} showed that this result holds on RH alone.

We can weaken the assumption that there are no zeros outside the box \eqref{strip} by using a strong zero-density hypothesis. Let $N(\sigma,T)$ denote the number of zeros $\rho=\beta +i\gamma$ with $\beta \ge \sigma$ and
$0<\gamma \le T$. 

\begin{theorem}\label{thm3} Assuming that 
\begin{equation}\label{Densitythm3} N(\sigma, T) = o\left( T^{2(1-\sigma)}\right) \qquad \text{for} \qquad \frac1{2} + \frac1{2\log T} \le \sigma \le \frac{25}{32}+\eta, \end{equation}
for any fixed $\eta>0$, then as $T\to \infty$, at least $61.7 \%$ of the nontrivial zeros of $\zeta(s)$ are simple. \end{theorem}

Selberg \cite{SelCollected2} made the conjecture that for all $\sigma>1/2$, we have
\begin{equation}\label{Selberg}
N(\sigma,T) = O\left( T^{1-c(\sigma-\frac12)} \frac{\log T}{\sqrt{\log \log T}}\right),
\end{equation}
where $c>0$ is some constant, and he stated that \eqref{Selberg} can often be used as a replacement for RH in the Selberg class. To see this, note that the conjecture is expected to hold for \emph{all} $\sigma>1/2$ and thus implies that almost all the nontrivial zeros are on the critical line $\{s\in\mathbb{C} : {\rm Re}(s)=1/2\}$. In a recent paper, Aryan \cite{Ary22} used this type of conjecture as a replacement of RH to obtain Montgomery's result on simple zeros. Our density conjecture \eqref{Densitythm3} implies all except $o(T)$ of the zeros are in the box \eqref{strip}, and if we extend this conjecture to all $\sigma>1/2$ we also obtain Montgomery's simple zero results. Iwaniec and Kowalski \cite[p. 249]{IK04} made a weaker Density Conjecture that $ N(\sigma,T) \ll T^{2(1-\sigma)} \log T$ for $1/2\le \sigma \le 1 $ and $T\ge 3$,
which, however, is too weak to be used in Theorem \ref{thm3}.

\section{Proof of Theorem \ref{thm1}}

Recall that if $\rho$ is a zero of $\zeta(s)$, then $1-\rho$, $\overline{\rho}$, and $1-\overline{\rho}$ are also zeros. Write 
\begin{equation*}
\rho = \beta +i\gamma := 1/2 +\delta + i\gamma,
\end{equation*}
where $-1/2<\delta <1/2$. If $\delta\neq 0$ and $\gamma >0$, then there is another zero in the upper half-plane given by
\[ 1-\overline{\rho} = 1/2 -\delta + i\gamma.\]
Therefore we may rewrite \eqref{F(x,T)} as 
\begin{equation}\label{F(x,T)2}
\begin{aligned}
F(x,T) = \sum_{\substack{\rho, \rho' \\ 0<\gamma,\gamma' \le T}} x^{\rho+\overline{\rho'}-1}w(\rho+\overline{\rho'}-1)= \sum_{\substack{\rho, \rho' \\ 0<\gamma,\gamma' \le T}} x^{\delta +\delta' +i(\gamma-\gamma')}w(\delta +\delta' +i(\gamma-\gamma')).
\end{aligned}
\end{equation}

\begin{lem}[Montgomery]\label{lem1} Let $\rho= 1/2+\delta +i\gamma$. Then for $x\ge 1$ and all $t$ we have
\begin{equation}\label{Explicit1}
\begin{split}
\sum_\rho \frac{ 2x^{\delta + i(\gamma-t)}}{1+((t-\gamma)+i\delta)^2} = 
-\sum_{n=1}^\infty\frac{\Lambda(n)}{n^{1/2+it}} & \min\big\{\frac{n}{x},\frac{x}{n}\big\} + x^{-1} \left(\log( |t|+2) + O(1) \right) \\ & + O\left(\frac{x^{1/2} }{1+t^2}\right) + O\left(\frac{x^{-5/2} }{|t|+2}\right).
\end{split}
\end{equation}
\end{lem}
\begin{proof}[Proof of Lemma \ref{lem1}] This is the Lemma from \cite{Montgomery73} if one takes $\sigma=3/2$ and $\delta=0$. 
The starting point for proving this lemma is the explicit formula due to Landau \cite{Lan09} that, for $x>1$ and $x\neq p^m$,
\begin{equation}\label{LandauExplicit}
\sum_\rho \frac{x^{\rho-s}}{s-\rho} = \sum_{n\le x}\frac{\Lambda(n)}{n^s} + \frac{\zeta'}{\zeta}(s) - \frac{x^{1-s}}{1-s} - \sum_{n=1}^\infty \frac{x^{-2n-s}}{2n+s},
\end{equation}
provided $s\neq 1$, $s \neq \rho$, $s\neq -2n$, which we henceforth assume. When $s=0$ this is the usual explicit formula for primes. Writing $s=\sigma+it$ and $\rho=1/2+\delta+i\gamma$, we obtain
\[ \sum_{\rho=1/2+\delta+i\gamma} \frac{x^{1/2+\delta-\sigma+i(\gamma-t)}}{\sigma-1/2-\delta+i(t-\gamma)} = R(\sigma+it), \]
where $R(\sigma+it)$ is the right-hand side of \eqref{LandauExplicit} which does not depend on $\rho$ and is treated exactly as in \cite{Montgomery73}.
Multiplying both sides by $x^{\sigma-1/2}$, we obtain
\[\sum_\rho \frac{x^{\delta+i(\gamma-t)}}{\sigma-1/2-\delta+i(t-\gamma)} = x^{\sigma-1/2}R(\sigma+it).\]
Next, replace $\sigma$ with $1-\sigma$ in the equation above which adds the conditions $s\neq 0$ and $s\neq 2n+1$ and gives
\[
\sum_\rho \frac{x^{\delta+i(\gamma-t)}}{1/2-\sigma-\delta+i(t-\gamma)} = x^{1/2-\sigma}R(1-\sigma +it). \]
Subtract this equation from the previous one and simplify to obtain
\[ \sum_\rho \frac{(2\sigma-1)x^{\delta+i(\gamma-t)}}{(\sigma - \frac12)^2+((t-\gamma)+i\delta)^2} = x^{\sigma-1/2}R(\sigma+it) - x^{1/2-\sigma}R(1-\sigma +it).\]
Taking $\sigma =3/2$ gives the left-hand side of the lemma, and the right-hand side is obtained exactly as in the original proof. 
\end{proof}

\begin{lem}\label{lem2} Letting $N(T)$ denote the number of zeros in the upper half plane up to height $T$, we have
\begin{equation}\label{N(T)} N(T) := \sum_{0 < \gamma \le T} 1= \frac{T}{2\pi} \log \frac{T}{2 \pi} - \frac{T}{2\pi} + O(\log T). \end{equation}
\end{lem}
\noindent This is proved in many books, for instance \cite[Theorem 25]{Ingham1932}, \cite[Theorem 9.4]{Titchmarsh}, or \cite[Corollary 14.3]{MontgomeryVaughan2007}. 
In particular, we have 
\begin{equation}\label{zeroestimte} N(T) \sim \frac{T}{2\pi}\log T ,\qquad \text{and}\qquad N(T+1)-N(T) \ll \log T. \end{equation}

\begin{lem}\label{lem3}
We have, for $x>0$ and $T\ge 3$,
\begin{equation}\label{lem3eq} F(x,T) = \frac{2}{\pi} \int_{-\infty}^{\infty} \left|\sum_{\substack{\rho \\ 0<\gamma \le T}} \frac{x^{\rho-1/2}}{1-\left(\rho-(1/2+it)\right)^2} \right|^2 \, dt. \end{equation}
\end{lem}

\begin{proof}[Proof of Lemma \ref{lem3}] We will make use of the formula, with $a\in \mathbb{C}$ and $-1< \text{Im}(a)<1$,
\begin{equation}\label{integral}
\begin{aligned}
\int_{-\infty}^\infty \frac{dt}{(1+t^2)(1+(t+a)^2)} &= 2\pi i \left( \frac1{2i(1+(a+i)^2)} + \frac1{2i(1+(-a+i)^2)} \right)= \frac{2\pi}{4+a^2},
\end{aligned}
\end{equation}
which is easily obtained by residues or Mathematica.
Now, multiplying out the right-hand side of \eqref{lem3eq}, we obtain 
\[ \begin{split} \frac{2}{\pi} \sum_{\substack{\rho, \rho' \\ 0<\gamma,\gamma' \le T}} x^{\rho+\overline{\rho'}-1}& \int_{-\infty}^\infty \frac{dt}{\big(1- \left(\rho -(1/2+it)\right)^2\big)\big(1- \left(\overline{\rho'} -(1/2-it)\right)^2\big)}\\&
=\frac{2}{\pi} \sum_{\substack{\rho, \rho' \\ 0<\gamma,\gamma' \le T}} x^{\rho+\overline{\rho'}-1}
\int_{-\infty}^\infty \frac{dt}{\big(1+ \left(t+i(\rho -1/2)\right)^2\big)\big(1+ \left(t-i(\overline{\rho'} -1/2)\right)^2\big)} \\&
=\frac{2}{\pi} \sum_{\substack{\rho, \rho' \\ 0<\gamma,\gamma' \le T}} x^{\rho+\overline{\rho'}-1}
\int_{-\infty}^\infty \frac{dt}{\big(1+ t^2\big)\big(1+ \left(t-i(\rho +\overline{\rho'} -1 )\right)^2\big)}\\&
= \sum_{\substack{\rho, \rho' \\ 0<\gamma,\gamma' \le T}} x^{\rho+\overline{\rho'}-1}\frac{4}{4+ \big(i(\rho +\overline{\rho'} -1)\big)^2}\\& 
= F(x,T)
\end{split}\]
by \eqref{F(x,T)2}.
\end{proof}

We now rewrite \eqref{lem3eq} with $\rho = 1/2+\delta +i\gamma$ so that
\begin{equation}\label{lem3eq2} F(x,T) = \frac{2}{\pi} \int_{-\infty}^{\infty} \left|\sum_{\substack{\rho \\ 0<\gamma \le T}} \frac{x^{\delta+i\gamma}}{1+ \left((t-\gamma)+i\delta\right)^2} \right|^2 \, dt. \end{equation}
Define 
\begin{equation}\label{Theta(t)} \Theta(t) := \max\left\{\beta : \rho = \beta +i\gamma,\ 0<|\gamma|\le t\right\}. \end{equation}
If RH is false then $\Theta(t)$ is a step function, and we often replace it with the \lq \lq zero-free" region in the complex $s=\sigma +i t$ plane
\begin{equation}\label{eta-zerofree} \sigma> 1-\eta(t) \ge \Theta(t), \qquad t\ge 3,\end{equation}
where $0<\eta(t)\le 1/2$ and $\eta(t)$ is a continuous decreasing but not necessarily strictly decreasing function. Clearly we can make $\eta(t)$ as close to $1-\Theta(t)$ as we wish pointwise except at the jumps of $\Theta(t)$. 
We will often make use of this with $\rho = 1/2+\delta + i\gamma$ in the form
\begin{equation}\label{delta-Theta} \delta \le \Theta(t)-1/2 \le 1/2-\eta(t) \qquad \text{for} \qquad 0<\gamma \le t.\end{equation}

Our next lemma gives the unconditional version of Montgomery's result that relates $F(x,T)$ to the explicit formula in Lemma \ref{lem1}.

\begin{lem}\label{lem4}
For $x\ge 1$ and $T\ge 3$, let 
\begin{equation}\label{L(x,T)} L(x,T) := \int_{0}^{T} \left|\sum_\rho \frac{2x^{\delta+i\gamma}}{1+ \left((t-\gamma)+i\delta\right)^2} \right|^2 \, dt. \end{equation}
Then
\begin{equation}\label{lem4eq} F(x,T) = \frac1{2\pi} L(x,T) + O\left( x^{1-2\eta(T\log^2T)}\log^3 T\right)+O(x). \end{equation}
\end{lem}

\begin{proof}[Proof of Lemma \ref{lem4}] We first truncate the sum over zeros in \eqref{L(x,T)} using trivial estimates, which then allows us to apply Montgomery's argument without modification to the truncated sum. For $x\ge 1$, since $|\delta|< 1/2$, we have the trivial estimate
\begin{equation}\label{trivialestimate} \left|\sum_\rho \frac{2x^{\delta+i\gamma}}{1+ \left((t-\gamma)+i\delta\right)^2 }\right|\ll x^{1/2}\sum_\gamma \frac1{1+(t-\gamma)^2} \ll x^{1/2}\log(|t|+2),\end{equation}
where the last estimate is well known and follows from the second estimate in \eqref{zeroestimte}.
In the same way we have, for $0\le t \le T$ and $Z\ge 2T$,
\begin{equation}\label{trivialestimate2} \left|\sum_{\substack{ \rho \\ |\gamma| \ge Z}}\frac{2x^{\delta+i\gamma}}{1+ \left((t-\gamma)+i\delta\right)^2 }\right|\ll x^{1/2}\sum_{\gamma\ge Z} \frac1{\gamma^2} \ll \frac{x^{1/2}\log Z}{Z}.\end{equation}

Next, on squaring we have
\[ L(x,T) = 4 \sum_{\rho, \rho'} x^{\delta+\delta'+i(\gamma-\gamma')} \int_0^T \frac{dt}{(1+ \left((t-\gamma)+i\delta\right)^2)(1+ \left((t-\gamma')-i\delta'\right)^2)}. \]
By \eqref{trivialestimate} and \eqref{trivialestimate2} we can exclude the terms with $\gamma \not \in [-Z,Z]$ in the sum above with an error 
\[\begin{split} &\ll \int_0^T \left|\sum_{\substack{ \rho \\ |\gamma| \ge Z}}\frac{x^{\delta+i\gamma}}{1+((t-\gamma)+i\delta))^2}\right|\left|\sum_{ \rho'}\frac{x^{\delta'-i\gamma'}}{1+((t-\gamma')-i\delta')^2}\right|\, dt 
\ll \frac{xT\log^2Z}{Z}.\end{split} \]
Taking $Z=T \log^2 T$, we conclude that
\[ L(x,T) = 4 \sum_{\substack{\rho, \rho'\\ |\gamma|\le Z, |\gamma'|\le Z}} x^{\delta+\delta'+i(\gamma-\gamma')} \int_0^T \frac{dt}{(1+ \left((t-\gamma)+i\delta\right)^2)(1+ \left((t-\gamma')-i\delta'\right)^2)} + O(x). \]
 Montgomery, arguing unconditionally except for taking $\delta=\delta'=0$ in $L(x,T)$ and with no truncation, showed the terms with $\gamma\not \in [0,T]$ can be excluded with an error $O(\log^3 T)$, and then the range of integration can be extended to $\mathbb{R}$ with an error $O(\log^2T)$. Here we apply the same argument where we need to include the factor 
 \begin{equation}\label{xdependence} x^{\delta +\delta' +i(\gamma-\gamma')} \ll x^{2\Theta(Z)-1}\le x^{1-2\eta(T\log^2 T)}\end{equation}
in the error term, at which point the bound for these error terms is majorized by dropping the truncation at $Z$ which then exactly matches Montgomery's argument. Thus we obtain by \eqref{lem3eq2}
\[\begin{split} L(x,T) &= 4 \sum_{\substack{\rho, \rho'\\ 0\le \gamma, \gamma'\le T} } x^{\delta+\delta'+i(\gamma-\gamma')} \int_{-\infty}^{\infty} \frac{dt}{(1+ \left((t-\gamma)+i\delta\right)^2)(1+ \left((t-\gamma')-i\delta'\right)^2)} \\&
\hskip 2in +O(x^{1-2\eta(T\log^2 T)}\log^3T)+ O(x) \\&
= 2\pi F(x,T) +O(x^{1-2\eta(T\log^2 T)}\log^3T)+ O(x).
\end{split}\]
\end{proof}

\begin{proof}[Proof of Theorem \ref{thm1}]
Since $w(u)$ is even, we see from \eqref{F(x,T)} that
$F(1/x,T) = F(x,T)$, and therefore $F(\alpha)$ is even. That $F(\alpha)$ is real and nonnegative follows immediately from Lemma \ref{lem3}.
 
We write \eqref{Explicit1} as $l(x,T)=r(x,T)$, and define
\begin{equation}\label{L=R} L(x,T) := \int_0^T\left|l(x,T)\right|^2\, dt = \int_0^T\left|r(x,T)\right|^2\, dt =: R(x,T). \end{equation}
As we just saw in Lemma \ref{lem4}, 
\[ L(x,T) = 2\pi F(x,T) +O(x^{1-2\eta(T\log^2 T)}\log^3T)+ O(x).\]
The current widest known zero-free region $\sigma \ge 1-\eta(t)$ was obtained independently by Korobov and Vinogradov with
\[ \eta(t) = \frac{c}{(\log{t})^{2/3}(\log\log{t})^{1/3}}, \qquad \text{for}~ t\ge3,\] 
for some constant $c>0$. Thus we see that, for $1\le x\le T^{1/2}$,
\[x^{1-2\eta(T\log^2 T)}\log^3T \ll T^{1/2}\log^3T,\]
while for $T^{1/2}\le x \le T$
\[x^{1-2\eta(T\log^2 T)}\log^3T \ll x \exp\left(-c\frac{\log x}{(\log{T})^{2/3}(\log\log{T})^{1/3}}\right)\log^3T \ll x. \] 
We conclude for $1\le x\le T$
\begin{equation}\label{Lfinal} L(x,T) = 2\pi F(x,T) + O(T) + O(x). \end{equation}

Next, $R(x,T)$ does not depend on RH, and Montgomery\cite{Montgomery73} proved unconditionally that 
\[ R(x,T) = (1+o(1)) T x^{-2} \log^2 T + T(\log x +O(1)) +O(x\log x). \]
From \cite{GM87} this was improved, so that for $0\le x\le T$,
\begin{equation}\label{Rfinal} R(x,T) = x^{-2} T\log T(\log T +O(1)) + T(\log x + O(\sqrt{\log T})). \end{equation}
Since $L(x,T)=R(x,T)$, \eqref{Lfinal} and \eqref{Rfinal} prove Theorem \ref{thm1} on using \eqref{MonF(alpha)} to convert $F(x,T)$ to $F(\alpha)$.
\end{proof}

\begin{remark}
In \eqref{Rfinal} we have removed an extraneous factor of $\log\log T$ from \cite{GM87} which can be avoided using Lemma 6 there in place of Lemma 7. The statement there has also been corrected slightly. See also \cite{Gold81} and \cite{LPZ17}. Montgomery and Vaughan in the forthcoming book Multiplicative Number Theory II have obtained a significantly refined version of {Montgomery's theorem}.
\end{remark}

\section{Sums over Differences of Zeros}

Let $g(\alpha)\in L^{\!1}(\mathbb R)$, and, for $z=x+iy$, $x,y \in \mathbb{R}$, define the Fourier transform $\widehat{g}(z)$ of $g(\alpha)$ by
\begin{equation}\label{hat-g}
\widehat{g}(z) = \int_{-\infty}^\infty g(\alpha)e(-z\alpha)\,d\alpha, \qquad \text{where} \qquad e(w) = e^{2\pi i w}.
\end{equation}
Thus $\widehat{g}(z)$ is an analytic function for all $z$. 
Taking $z= i(\rho-\rho')\frac{\log T}{2\pi}$, we have 
\[ \widehat{g}\left( i(\rho-\rho')\frac{\log T}{2\pi}\right) = \int_{-\infty}^\infty g(\alpha)T^{\alpha(\rho -\rho')}\,d\alpha.
\]
Multiplying both sides of this equation by $w(\rho-\rho')$ and summing over $0<\gamma, \gamma' \le T$, we obtain
\begin{equation}\label{pairsum}
\sum_{\substack{\rho, \rho' \\ 0<\gamma,\gamma'\le T}} \widehat{g}\left(i(\rho -\rho')\frac{\log T}{2\pi}\right) w(\rho-\rho') 
= \left(\frac{T}{2\pi}\log T\right)\int_{-\infty}^\infty F(\alpha)g(\alpha)\,d\alpha.
\end{equation}

We now apply Theorem \ref{thm1} to obtain the following unconditional version of Montgomery's result on evaluating sums over pairs of zeros for even kernels with Fourier transforms supported in $[-1,1]$. If we assume RH this agrees with the earlier version in \cite{Montgomery73}. 

\begin{lem}\label{lem5} Suppose $ \alpha \in \mathbb{R}$ and $z\in \mathbb{C}$. Suppose $r(\alpha)$
is a real-valued even function in $L^{\!1}(\mathbb R)$ with support in $[-1,1]$, and also $r(\alpha)$ is Lipschitz continuous at $\alpha =0$. Then $\widehat r(z)$ is an even analytic function,
\begin{equation}\label{rhateven}
\widehat{r}(z) = 2\int_{0}^1 r(\alpha) \cos(2 \pi z\alpha)\,d\alpha, \end{equation}
and we have
\begin{equation}\label{lem5eq}
\sum_{\substack{\rho, \rho' \\ 0<\gamma,\gamma'\le T}} \widehat{r}\left(i(\rho -\rho')\frac{\log T}{2\pi}\right) w(\rho-\rho') = \frac{T}{2\pi}\log T
\left(r(0) + 2\int_{0}^1 \alpha r(\alpha)\,d\alpha+O\left(\frac1{\sqrt{\log T}}\right)\right).
\end{equation}
\end{lem}
Recall that a function $f(x)$ is Lipschitz continuous at a point $x=a$ if there are constants $C>0$ and $\delta >0$ such that $|f(x)-f(a)| \le C |x-a|$ for all $x$ in a neighborhood $|x-a|< \delta$ of $a$. 

\begin{proof}[Proof of Lemma \ref{lem5}] In \eqref{pairsum} we take $g(\alpha)= r(\alpha)$, and see that \eqref{rhateven}
follows from \eqref{hat-g} by the evenness of $r$. For the integral in \eqref{pairsum}, we apply \Cref{thm1} and have
\[ \begin{split} \int_{-\infty}^\infty F(\alpha)r(\alpha)\,d\alpha &= 2\int_0^1 \left(T^{-2\alpha}(\log T +O(1))+ \alpha +O\left(\frac1{\sqrt{\log T}}\right)\right)r(\alpha) \, d\alpha\\&
=2\int_0^1 T^{-2\alpha}(\log T +O(1)) r(\alpha)\, d\alpha + 2\int_0^1\alpha r(\alpha)\, d\alpha +O\left(\frac1{\sqrt{\log T}}\right). \end{split}\]
To complete the proof, we use of the Lipschitz condition on $r(\alpha)$ at $\alpha=0$ to see the first integral is
\[ \begin{split} &= 2\int_0^{\log\log T/\log T} T^{-2\alpha}(\log T +O(1))r(\alpha) \, d\alpha+O\left( \int_{\log\log T/\log T}^1 T^{-2\alpha}\log T |r(\alpha)|\, d\alpha \right) \\&
= 2\left( r(0)+ O\left(\frac{\log\log T}{\log T}\right)\right) \int_0^{\log\log T/\log T}T^{-2\alpha}(\log T +O(1))\, d\alpha + O\left(\frac1{\log T}\int_0^1|r(\alpha )|\, d\alpha\right)\\&
= 2\left( r(0)+ O\left(\frac{\log\log T}{\log T}\right)\right)\left(\frac12 +O\left(\frac1{\log T}\right)\right) + O\left(\frac1{\log T}\right) \\&
= r(0) + O\left(\frac{\log\log T}{\log T}\right).
\end{split}\]
\end{proof}

\section{Tsang's Kernel}
We define the Tsang kernel $K(z)$ through its Fourier transform by
\begin{equation}\label{TsangFT} \widehat{K}(t) := j(2\pi t){\rm sech}(2 \pi t) ,\end{equation}
where $j(\alpha)$ is an even, non-negative, bounded function supported on $|\alpha|\le 1$, and we also assume $j$ is twice differentiable on $[0,1]$ with one-sides derivatives at the endpoints. We also require, for all $w \in \mathbb{R}$,
\[ 0\le \widehat{j}(w) \ll \frac1{1+w^2}. \]

Thus the Tsang kernel $K(z)$ is
\begin{equation}\label{K(z)} K(z)= \int_{-\infty}^\infty \widehat{K}(t) e(zt)\, dt = 2\int_0^\infty j(2\pi t){\rm sech}(2\pi t)\cos(2\pi z t) \, dt 
=\frac1{\pi}\int_0^1 j(\alpha){\rm sech}( \alpha)\cos( z \alpha) \, d\alpha.\end{equation}

Tsang took the function $\widehat j$ to be the Fej\'er kernel, that is, 
\begin{equation}\label{Fejer}j_F(\alpha)=\max\{0,1-|\alpha|\}, \qquad \widehat j_F(w)= \left(\frac{\sin \pi w}{\pi w}\right)^2, \end{equation}
and we will also take $\widehat j$ to be the Montgomery-Taylor kernel \cite{Montgomery74}, \cite{CheerGold} given by
\begin{equation}\label{ M-T}
j_M(\alpha) = \frac1{1 - \cos{\sqrt{2}}}\left(\frac1{2\sqrt{2}}\sin\left(\sqrt{2}j_F(\alpha))\right) + \frac12j_F(\alpha)\cos\left(\sqrt{2}\alpha\right)\right), \end{equation}
where $j_F(\alpha)$ is given in \eqref{Fejer}, and 
\begin{equation}
\widehat j_M(w) = \frac1{1 - \cos{\sqrt{2}}}\left(\frac{\sin(\frac12(\sqrt{2}-2\pi w))}{\sqrt{2}-2\pi w} + \frac{\sin(\frac12(\sqrt{2}+2\pi w))}{\sqrt{2}+2\pi w} \right)^2.
\end{equation}

Using the properties of $j$, Tsang proved that $K$ has the following properties \cite[Lemma 1]{Tsang3}. 
\begin{lem}[K.-M. Tsang]\label{lemTsang} The kernel $K(z)$ is an even entire function such that:
\begin{enumerate}[label={(\alph*)}]
\item\label{lemTsang-a} $K(x)>0$ for all $x \in \mathbb{R}$,

\item\label{lemTsang-b} For $z\in\mathbb{C}-\{0\}$, $K(z) \ll \frac{e^{|{\rm Im}(z)|}}{|z|^2}$,

\item\label{lemTsang-c} For $z=x+iy$, $x,y\in \mathbb{R}$, then when $|y|<1$, we have ${\rm Re}\,K(x+i y)>0$.
\end{enumerate}
\end{lem}

\section{Application to sums over differences of zeros}
\label{sec-sum-diff-zeros}

Applying the Tsang kernel in Lemma \ref{lem5} we obtain the following result. 
\begin{lem}\label{lem7} We have
\begin{equation}\label{lem7a} 2\pi \sum_{\substack{\rho, \rho' \\ 0<\gamma,\gamma'\le T\\ |\beta-\beta'|<\frac1{\log T}}} 
{\rm Re}\,K\left(-i(\rho -\rho')\log{T}\right) + \mathcal{S}(T)=
\left(\widehat{K}(0) + 2\int_0^1 \alpha\widehat{K}\left(\frac{\alpha}{2\pi}\right)\, d\alpha + O\left(\frac1{\sqrt{\log T}}\right)\right)\frac{T}{2\pi}\log T,
\end{equation}
where
\begin{equation}\label{lem7b} \mathcal{S}(T) := 2\pi {\rm Re}\sum_{\substack{\rho, \rho' \\ 0<\gamma,\gamma'\le T\\ |\beta-\beta'| \ge \frac1{\log T}}} K\left(-i(\rho -\rho')\log{T}\right)w(\rho-\rho') ,
\end{equation}
and $K$ and $\widehat K$ are given in \eqref{TsangFT} and \eqref{K(z)}. Here ${\rm Re}\,K> 0$ for every term in the sum in \eqref{lem7a}, and $\widehat{K}\ge 0$.
\end{lem}
\begin{proof}[Proof of Lemma \ref{lem7}]
In Lemma \ref{lem5} we take, for $\alpha \in \mathbb{R}$,
$$ r(\alpha) = \widehat{K}\left(\frac{\alpha}{2\pi}\right) = j(\alpha){\rm sech}(\alpha), $$
so that ${\rm supp}\,r\subset[-1,1]$ and
$$ \widehat{r}(z) = 2\pi K(-2\pi z). $$
Thus \eqref{lem5eq} becomes 
\begin{equation}\label{kernel-sum}
2\pi \sum_{\substack{\rho, \rho' \\ 0<\gamma,\gamma'\le T}} 
K\left(-i(\rho -\rho')\log{T}\right) w(\rho-\rho')
= \left(\widehat{K}(0) + 2\int_0^1 \alpha\widehat{K}\left(\frac{\alpha}{2\pi}\right)\, d\alpha +O\left(\frac1{\sqrt{\log T}}\right)\right)\frac{T}{2\pi}\log T.
\end{equation} 
From Lemma \ref{lemTsang} we have
\begin{equation}\label{pos-real}
{\rm Re}\,K\left(-i(\rho -\rho')\log{T}\right)
= {\rm Re}\,K\left((\gamma-\gamma'-i(\beta-\beta'))\log T\right) > 0 \qquad \text{if} \quad |\beta-\beta'| < 1/\log T.
\end{equation}

Since $w(\rho-\rho')$ is complex-valued for zeros off the half-line, we first need to remove this weight when $|\beta-\beta'| < \frac1{\log T}$ before applying \eqref{pos-real}. This is easily done
since, noting $w(z)-1 = \frac14z^2w(z)$ and so $w(0)-1=0$, and using \ref{lemTsang-b} of Lemma \ref{lemTsang},
\begin{equation}\label{removing-weight}
\begin{aligned}
\sum_{\substack{\rho,\rho' \\ 0<\gamma,\gamma'\le T\\ |\beta-\beta'| < \frac1{\log T}}}
K\left(-i(\rho-\rho')\log{T}\right)
(w(\rho-\rho')-1) &=
\sum_{\substack{\rho\neq\rho' \\ 0<\gamma,\gamma'\le T\\ |\beta-\beta'| < \frac1{\log T}}}
K\left(-i(\rho-\rho')\log{T}\right)
(w(\rho-\rho')-1)
\\ & \ll \sum_{\substack{\rho\neq\rho' \\ 0<\gamma,\gamma'\le T\\ |\beta-\beta'| < \frac1{\log T}}} \frac{T^{|\beta-\beta'|}}{|\rho-\rho'|^2\log^2T} \frac{|\rho-\rho'|^2}{|4-(\rho-\rho')^2|} \\
&\le \frac1{\log^2T} \sum_{\substack{\rho\neq\rho' \\ 0<\gamma,\gamma'\le T}} \frac{T^{1/\log T}}{|4-(\rho-\rho')^2|} \\ &\ll \frac1{\log^2T} \sum_{0<\gamma,\gamma'\le T} \frac1{1+(\gamma-\gamma')^2} 
\ll T, 
\end{aligned}
\end{equation}
where the last line uses Lemma \ref{lem2} as is done for the sum over zeros in \eqref{trivialestimate}. We now remove $w(\rho-\rho')$ from the terms in \eqref{kernel-sum} with $|\beta-\beta'| < \frac1{\log T}$ and take real parts to complete the proof.
\end{proof}

\begin{remark} In applications we only need the bound in \eqref{removing-weight} to be $o(T\log{T})$, which is obtained if we sum over all pairs of nontrivial zeros $\rho,\rho'$ with $|\beta-\beta'|<\frac{(1-\epsilon)\log\log{T}}{\log{T}}$ for any $\epsilon>0$.
\end{remark}

\section{Assumptions on Zeros}

In Theorem \ref{thm2} we assume all the zeros $\rho = \beta +i\gamma$ of the Riemann zeta-function with $T^{3/8}<\gamma\le T$ lie within the thin box
\begin{equation}\label{strip2} \frac1{2} -\frac1{2\log T}< \beta < \frac1{2} + \frac1{2\log T}, \end{equation}
and in Theorem \ref{thm3} we assume the strong density hypothesis
\begin{equation}\label{Densitythm3B} N(\sigma, T) = o\left( T^{2(1-\sigma)}\right) \qquad \text{for} \qquad \frac1{2} + \frac1{2\log T} \le \sigma \le \frac{25}{32}+\eta, \end{equation}
for any fixed $\eta>0$. 

We now prove that either of these assumptions implies that,
 for any sufficiently large $T$,
\begin{equation} \label{S(T)bound} \mathcal{S}(T) = 2\pi {\rm Re}\sum_{\substack{\rho, \rho' \\ 0<\gamma,\gamma'\le T\\ |\beta-\beta'| \ge \frac1{\log T}}} K\left(-i(\rho -\rho')\log{T}\right)w(\rho-\rho') = o(T\log T). \end{equation}
To do this, we first prove that the density hypothesis \eqref{Densitythm3B} implies \eqref{S(T)bound}, and next show that the essentially stronger hypothesis \eqref{strip2} implies \eqref{Densitythm3B}, and thus \eqref{S(T)bound} by the first step. We use standard results and methods for applying zero-density results to explicit formulas, see \cite[Chapter 12]{ivic} and \cite[Chapter 10]{IK04}.

By property \ref{lemTsang-b} of Lemma \ref{lemTsang}, we have
$$ K(z) \ll \frac{e^{|{\rm Im} (z)|}} {|z|^2}, $$
and therefore
$$
K(-i(\rho -\rho')\log T)
\ll \frac{T^{|\beta-\beta'|}} {((\beta-\beta')\log T)^2+((\gamma-\gamma')\log T)^2}.
$$
Hence since $|w(\rho-\rho')|\ll 1$, we have 
\[ \mathcal{S}(T) 
\ll \sum_{\substack{\rho, \rho' \\ 0<\gamma,\gamma'\le T\\ |\beta-\beta'| \ge \frac1{\log T}}} \frac{T^{|\beta-\beta'|}} {((\beta-\beta')\log T)^2+((\gamma-\gamma')\log T)^2}.\]
By the inequality $|ab|\le \frac12(a^2 +b^2)$, we have 
\[ T^{|\beta-\beta'|} = T^{|(\beta-1/2)+(1/2-\beta')|}\le T^{|\beta-1/2|} T^{|\beta'-1/2|}\le \frac12\left(T^{|2\beta-1|}+ T^{|2\beta'-1|}\right). \]
Thus
\[ \mathcal{S}(T) 
\ll \sum_{\substack{\rho, \rho' \\ 0<\gamma,\gamma'\le T\\ |\beta-\beta'| \ge \frac1{\log T}}} \frac{T^{|2\beta-1|}+ T^{|2\beta'-1|}} {((\beta-\beta')\log T)^2+((\gamma-\gamma')\log T)^2}.\]
Since $|\beta-\beta'| \ge \frac1{\log T}$, at least one of $\beta$ or $\beta'$ must be outside the interval 
$$\left(\frac12-\frac1{2\log T},\ \frac12+\frac1{2\log T}\right). $$
By relabeling if necessary, we may assume that $\beta$ is outside this interval and that $T^{|2\beta'-1|}\leq T^{|2\beta-1|}$. Hence
\[ \begin{split} \mathcal{S}(T) &\ll \sum_{\substack{\rho=\beta+i\gamma \\ |\beta-1/2| \ge \frac1{2\log T}\\ 0<\gamma \le T}} T^{|2\beta-1|} \sum_{\substack{0<\gamma'\le T \\ |\beta-\beta'| \ge 1/\log T }} \frac1 {((\beta-\beta')\log T)^2+((\gamma-\gamma')\log T)^2}\\&
\ll \sum_{\substack{|\beta-1/2| \ge \frac1{2\log T}\\ 0<\gamma \le T}} T^{|2\beta-1|} \sum_{0<\gamma'\le T} \frac1 {1+((\gamma-\gamma')\log T)^2}. \end{split} \]
If $\rho$ is a zero of zeta then so is $1-\rho$, and $|2\beta-1|=|2(1-\beta)-1|$. Hence
\[ \mathcal{S}(T) \ll \sum_{\substack{\frac12 + \frac1{2\log T} \le \beta< 1\\ 0<\gamma \le T}} T^{|2\beta-1|} \sum_{0<\gamma'\le T} \frac1 {1+((\gamma-\gamma')\log T)^2}. \]
By \eqref{zeroestimte} we have
\begin{equation}\label{S(T)inner_sum}
\begin{aligned}
\sum_{0<\gamma'\le T} \frac1 {1 +((\gamma-\gamma')\log T)^2} &\ll \int_3^T 
\frac {\log t} {1+((t-\gamma)\log T)^2}\, dt \\&
\le \int_{-\infty}^{\infty} \frac {1} {1+ v^2}\, dv 
= \pi,
\end{aligned}
\end{equation}
and therefore we have 
\[ \mathcal{S}(T) \ll \sum_{\substack{\frac12 + \frac1{2\log T} \le \beta \le 1 \\ 0<\gamma \le T}} T^{2\beta-1}. \]
Applying Bourgain's zero-density estimate \cite{Bou00}
\begin{equation}\label{Bourgain-density}
N(\sigma,T) = o(T^{2(1-\sigma)}) \qquad\text{for}\quad 25/32+\eta\le\sigma\le1,
\end{equation} which is the hypothesis \eqref{Densitythm3B} in the remaining range. Hence
\[ \begin{split} \mathcal{S}(T) &\ll \int_{\frac12 + \frac1{2\log T}}^1 T^{2u-1}\, dN(u,T) \\&
= T^{2/\log T} N\left(\frac12 + \frac1{2\log T},T\right) - 2\log T\int_{\frac12 + \frac1{2\log T}}^1 N(u,T) T^{2u-1}\, du \\& =o(T\log T),
\end{split} \]
which proves \eqref{S(T)bound}.

We now prove that the assumption \eqref{strip2} implies \eqref{Densitythm3B}. Let $0<\eta<\frac{1}{32}$ be arbitrary and fixed, and suppose that
$$ \frac{1}{2}+\frac{1}{2\log T} \leq \sigma \leq \frac{25}{32}+\eta. $$
Then by our hypothesis we have
\begin{align*}
N(\sigma,T)
&= \#\{\rho =\beta+i\gamma : \beta\geq \sigma \text{ and } T^{3/8}<\gamma\leq T\} + \#\{\rho =\beta+i\gamma : \beta\geq \sigma \text{ and } 0<\gamma\leq T^{3/8}\} \\
&= \#\{\rho =\beta+i\gamma : \beta\geq \sigma \text{ and } 0<\gamma\leq T^{3/8}\} \\
&\leq \#\{\rho =\beta+i\gamma : 0<\gamma\leq T^{3/8}\}.
\end{align*}
The number on the last line is $N(T^{3/8})$ which is $O(T^{3/8}\log T)$ by Lemma \ref{lem2}.
Thus
\begin{align*}
N(\sigma,T) \ll T^{\frac{3}{8}+\varepsilon} 
= T^{2(1-\frac{26}{32})+\varepsilon} 
= o(T^{2(1-\sigma)})
\end{align*}
as $T\rightarrow \infty$, since
$$ 1-\frac{26}{32} < 1-\frac{25}{32}-\eta \leq 1-\sigma. $$
Hence
$$ N(\sigma,T)=o(T^{2(1-\sigma)}) \qquad\text{for}\quad \frac{1}{2}+\frac{1}{2\log T} \leq \sigma \leq \frac{25}{32}+\eta \quad\text{and fixed } 0<\eta<\frac{1}{32}, $$
as $T\rightarrow \infty$.
The estimate $N(\sigma,T)=o(T^{2(1-\sigma)})$ also holds for $\sigma\geq \frac{25}{32}+\varepsilon$, for any $\varepsilon>0$, by \cite[p. 146]{Bou00}. Therefore
$$ N(\sigma,T)=o(T^{2(1-\sigma)}) \qquad\text{for}\quad \frac{1}{2}+\frac{1}{2\log T} \leq \sigma \leq \frac{25}{32}+\eta \quad\text{and fixed } \eta>0, $$
as $T\rightarrow \infty$, which is \eqref{Densitythm3B}.

\section{Proof of Theorem \ref{thm2} and Theorem \ref{thm3}}
Letting $m_\rho$ denote the multiplicity of a zero $\rho$ of $\zeta(s)$, then
$$
\sum_{\substack{\rho \\ 0<\gamma\le T}} m_\rho = \sum_{\substack{\rho, \rho' \\ 0<\gamma,\gamma'\le T \\ \rho=\rho'}} 1
= \frac1{K(0)} \sum_{\substack{\rho=\rho' \\ 0<\gamma,\gamma'\le T}} {\rm Re}\,K\left(-i(\rho -\rho')\log{T}\right).
$$
Next note trivially that if $\rho=\rho'$, then the zeros are within the range $|\beta-\beta'|<\frac1{\log{T}}$.
By \ref{lemTsang-c} of \Cref{lemTsang}, we also have that ${\rm Re}\,K\left(-i(\rho -\rho')\log{T}\right)>0$ in the same range.
Therefore we may upper bound the sum on the right-hand side above by extending the sum to all zeros with $|\beta-\beta'|<\frac1{\log{T}}$ and obtain
\begin{equation*}
\begin{aligned}
\sum_{\substack{\rho \\ 0<\gamma\le T}} m_\rho
&\le \frac1{K(0)} \sum_{\substack{\rho,\rho' \\ |\beta-\beta'|<\frac1{\log{T}} \\ 0<\gamma,\gamma'\le T}} {\rm Re}\,K\left(-i(\rho -\rho')\log{T}\right).
\end{aligned}
\end{equation*}

We proved in the last section that the assumption on zeros used in either one of Theorem \ref{thm2} or Theorem \ref{thm3} implies that
$\mathcal{S}(T) = o(T\log{T})$. 
Therefore \eqref{lem7a} of \Cref{lem7} gives
\begin{equation*}
\frac1{K(0)} \sum_{\substack{\rho,\rho' \\ |\beta-\beta'|<\frac1{\log{T}} \\ 0<\gamma,\gamma'\le T}} {\rm Re}\,K\left(-i(\rho -\rho')\log{T}\right)
\sim \frac1{2\pi K(0)}\left(\widehat{K}(0) + 2\int_0^1 \alpha\widehat{K}\left(\frac{\alpha}{2\pi}\right)\, d\alpha \right)\frac{T}{2\pi}\log T,
\end{equation*}
which implies
\begin{equation}\label{sum-multiplicity}
\sum_{\substack{\rho \\ 0<\gamma\le T}} m_\rho
\le \frac1{2\pi K(0)}\left(\widehat{K}(0) + 2\int_0^1 \alpha\widehat{K}\left(\frac{\alpha}{2\pi}\right)\, d\alpha + o(1) \right)\frac{T}{2\pi}\log T,
\end{equation}
as $T\to \infty$.
Following Montgomery's \cite{Montgomery73} argument, we see that the number of zeros which are simple satisfies
\begin{equation*}
\sum_{\substack{\rho:\,\text{simple} \\ 0<\gamma\le T}} 1 \ge \sum_{\substack{\rho \\ 0<\gamma\le T}} (2-m_\rho).
\end{equation*}
Hence, the proportion of simple zeros of $\zeta(s)$ is
\begin{equation*}
\frac1{N(T)}\sum_{\substack{\rho:\,\text{simple} \\ 0<\gamma\le T}} 1 \ge 2 - \frac1{N(T)} \sum_{\substack{\rho \\ 0<\gamma\le T}} m_\rho
\end{equation*}
which, since $N(T)\sim \frac{T}{2\pi}\log T$ by Lemma \ref{lem2} gives from \eqref{sum-multiplicity} 
\begin{equation}\label{simple-proportion}
\frac1{N(T)}\sum_{\substack{\rho:\,\text{simple} \\ 0<\gamma\le T}} 1 \ge 2 - \frac1{2\pi K(0)} \left(\widehat{K}(0) + 2\int_0^1 \alpha\widehat{K}\left(\frac{\alpha}{2\pi}\right)\, d\alpha + o(1)\right).
\end{equation}


Suppose first we take the Fejer kernel $j(\alpha) = j_F(\alpha)$. Then $\widehat{K}(0)=1$ and computation gives
$$ 2\int_0^1 \alpha\widehat{K}\left(\frac{\alpha}{2\pi}\right)\, d\alpha
= 2\int_0^1 \alpha\left(1-\alpha\right){\rm sech}\,\alpha~ d\alpha = 0.2913876354\ldots. $$
Further, applying \eqref{K(z)}, we have upon computation that
\[ \pi K(0) = \int_0^1 j(u) {\rm sech}(u)\,du
= \int_0^1 (1-u) {\rm sech}(u)\,du = 0.4640648392\ldots. \]
Substituting these into \eqref{simple-proportion} we have
\begin{align*}
\frac1{N(T)}\sum_{\substack{\rho:\,\text{simple} \\ 0<\gamma\le T}} 1
\ge 2 - \frac{1.291387636 + o(1)}{2\times0.464064839}
= 0.608612927\ldots + o(1).
\end{align*}
Thus if all the nontrivial zeros $\rho$ of $\zeta(s)$ lie within the box $|\beta-1/2|<\frac1{2\log T}$, $T^{3/8}<\gamma \le T$, then at least $60.8\%$ of them are simple.


We improve the above proportion to $61.7\%$ using the Montgomery-Taylor kernel $j(\alpha) =j_M(\alpha)$.
Computation gives
$$ \widehat{K}(0)= j_M(0) = 1.0061271908\ldots, $$
$$ 2\int_0^1 \alpha\widehat{K}\left(\frac{\alpha}{2\pi}\right)\, d\alpha
= 2\int_0^1 \alpha j_M(\alpha){\rm sech}\,\alpha~ d\alpha = 0.2832624869\ldots, $$
and
\[ \pi K(0) = \int_0^1 j_M(u) {\rm sech}(u)\,du = 0.4663199124\ldots. \]
Hence substituting these values into \eqref{simple-proportion} as before, we have
\begin{align*}
\frac1{N(T)}\sum_{\substack{\rho:\,\text{simple} \\ 0<\gamma\le T}} 1
\ge 2 - \frac{1.289389678 + o(1)}{2\times0.466319912}
= 0.617483786\ldots + o(1).
\end{align*}

\section*{Acknowledgement and Funding}
The authors thank the American Institute of Mathematics for its hospitality and for providing a pleasant research environment where most of the research was conducted. The first author is supported by NSF DMS-1854398 FRG. The third author was supported by JSPS KAKENHI Grant Number 22K13895. The fourth author is partially supported by NSF DMS-1902193 and NSF DMS-1854398 FRG.

\bibliographystyle{alpha}
\bibliography{AHBibliography}

\begin{thebibliography}{CGdL20}

\bibitem[Ary22]{Ary22}
F.~Aryan.
\newblock On an extension of the {L}andau-{G}onek formula.
\newblock {\em J. Number Theory}, 233:389--404, 2022.

\bibitem[BHB13]{BuiHB}
H.~M. Bui and D.~R. Heath-Brown.
\newblock On simple zeros of the riemann zeta-function.
\newblock {\em Bull. London Math. Soc.}, 45(5):953--961, 2013.

\bibitem[Bou00]{Bou00}
Jean Bourgain.
\newblock On large values estimates for {D}irichlet polynomials and the density
  hypothesis for the {R}iemann zeta function.
\newblock {\em Internat. Math. Res. Notices}, 2000(3):133--146, 2000.

\bibitem[CG93]{CheerGold}
A.~Y. Cheer and D.~A. Goldston.
\newblock Simple zeros of the {R}iemann zeta-function.
\newblock {\em Proc. Amer. Math. Soc.}, 118(2):365--372, 1993.

\bibitem[CGdL20]{CGL2020}
Andrés Chirre, Felipe Gonçalves, and David de~Laat.
\newblock Pair correlation estimates for the zeros of the zeta function via
  semidefinite programming.
\newblock {\em Adv. Math.}, 361(106926):22 pp., 2020.

\bibitem[CGG98]{CGG98}
J.~B. Conrey, A.~Ghosh, and S.~M. Gonek.
\newblock Simple zeros of the riemann zeta-function.
\newblock {\em Proc. London Math. Soc.}, 76(3):497--522, 1998.

\bibitem[CIS13]{CIS13}
J.~Brian Conrey, Henryk Iwaniec, and Kannan Soundararajan.
\newblock Critical zeros of {D}irichlet {$L$}-functions.
\newblock {\em J. Reine Angew. Math.}, 681:175--198, 2013.

\bibitem[GM87]{GM87}
D.~A. Goldston and H.~L. Montgomery.
\newblock Pair correlation of zeros and primes in short intervals.
\newblock In {\em Analytic Number Theory and Diophantine Problems (A. C.
  Adolphson and et al., eds.)}, volume~70 of {\em Progr. Math., Proc. of a
  Conference at Oklahoma State University (1984)}, pages 183--203. Birkhauser
  Verlag, 1987.

\bibitem[Gol81]{Gold81}
D.~A. Goldston.
\newblock Large differences between consecutive prime numbers.
\newblock {\em Thesis}, U. C. Berkeley:1--75, 1981.

\bibitem[IK04]{IK04}
H.~Iwaniec and E.~Kowalski.
\newblock {\em Analytic number theory}, volume~53 of {\em Amer. Math. Soc.
  Colloq. Publ.}
\newblock American Mathematical Society, Providence, RI, 2004.

\bibitem[Ing90]{Ingham1932}
A.~E. Ingham.
\newblock {\em The Distribution of Prime Numbers}.
\newblock Cambridge Mathematical Library, Cambridge University Press,
  Cambridge, 1990.
\newblock Reprint of the 1932 original; With a foreword by R. C. Vaughan.

\bibitem[Ivi85]{ivic}
A.~Ivi\'c.
\newblock {\em The {R}iemann zeta-function}.
\newblock John {W}iley \& {S}ons, New York, 1985.

\bibitem[Lan09]{Lan09}
Edmund Landau.
\newblock {\em Handbuch der {L}ehre von der {V}erteilung der {P}rimzahlen}.
\newblock Teubner, Berlin, 1909.

\bibitem[LPZ17]{LPZ17}
A.~Languasco, A.~Perelli, and A.~Zaccagnini.
\newblock An extended pair correlation conjecture and primes in short
  intervals.
\newblock {\em Trans. Amer. Math. Soc.}, 369:4235--4250, 2017.

\bibitem[Mon73]{Montgomery73}
H.~L. Montgomery.
\newblock The {P}air {C}orrelation of {Z}eros of the {Z}eta {F}unction.
\newblock In {\em Analytic Number Theory ({P}roc. {S}ympos. {P}ure {M}ath.,
  {V}ol. {XXIV}, {S}t. {L}ouis {U}niv., {S}t. {L}ouis, {M}o., 1972)}, pages
  181--193. Amer. Math. Soc., Providence, R.I., 1973.

\bibitem[Mon74]{Montgomery74}
H.~L. Montgomery.
\newblock Distribution of zeros of the riemann zeta function.
\newblock In {\em Proc. Int. Cong. Math. Vancouver, 1974}, pages 379--381. IMU,
  1974.

\bibitem[MV07]{MontgomeryVaughan2007}
H.~L. Montgomery and R.~C. Vaughan.
\newblock {\em Multiplicative Number Theory I: Classical Theory}, volume~97 of
  {\em Cambridge Studies in Advanced Mathematics}.
\newblock Cambridge University Press, Cambridge, 2007.

\bibitem[PRZZ20]{Pra20}
Kyle Pratt, Nicolas Robles, Alexandru Zaharescu, and Dirk Zeindler.
\newblock More than five-twelfths of the zeros of $\zeta$ are on the critical
  line.
\newblock {\em Res. Math. Sci.}, 7(2):74pp, 2020.

\bibitem[Sel91]{SelCollected2}
A.~Selberg.
\newblock {\em Old and new conjectures and results about a class of Dirichlet
  series, Collected papers. {V}ol. {II}}.
\newblock Springer-Verlag, Berlin, 1991.
\newblock With a foreword by K. Chandrasekharan.

\bibitem[Tit86]{Titchmarsh}
E.~C. Titchmarsh.
\newblock {\em The theory of the {R}iemann zeta-function}.
\newblock The Clarendon Press, Oxford University Press, New York, second
  edition, 1986.
\newblock Edited and with a preface by D. R. Heath-Brown.

\bibitem[Tsa93]{Tsang3}
K.-M. Tsang.
\newblock The large values of the {R}iemann zeta-function.
\newblock {\em Mathematika}, 40(2):203--214, 1993.

\end{thebibliography}

\end{document}